\documentclass[12pt, reqno]{amsart}
\usepackage{amsmath, amsthm, amscd, amsfonts, amssymb, graphicx, color}
\usepackage[bookmarksnumbered, colorlinks, plainpages]{hyperref}
\hypersetup{colorlinks=true,linkcolor=red, anchorcolor=green, citecolor=cyan, urlcolor=red, filecolor=magenta, pdftoolbar=true}

\newtheorem{theorem}{Theorem}[section]
\newtheorem{lemma}[theorem]{Lemma}

\theoremstyle{definition}
\newtheorem{definition}[theorem]{Definition}
\newtheorem{example}[theorem]{Example}

\theoremstyle{remark}

\numberwithin{equation}{section}

\begin{document}
\setcounter{page}{1}
\title[]{$\ast$-g-frames in tensor products of hilbert $C^{\ast}$-modules}

	\author{Mohamed ROSSAFI$^{*}$}
	\address{Department of Mathematics
		\newline \indent University of Ibn Tofail
		\newline \indent B.P.133 Kenitra, Morocco}
	\email{rossafimohamed@gmail.com}

\author{ Samir KABBAJ}
\address{Department of Mathematics
	\newline \indent University of Ibn Tofail
	\newline \indent B.P.133 Kenitra, Morocco}
\email{samkabbaj@yahoo.fr }

\date{19/07/2017}


\begin{abstract}
In this paper, we study $\ast$-g-frames in tensor products of Hilbert $C^{\ast}$-modules. We show that a tensor product of two $\ast$-g-frames is a $\ast$-g-frames, and we get some result.
\end{abstract} \maketitle
\renewcommand{\thefootnote}{}
\footnotetext{ $^*$Corresponding author
	\par
	E-mail address: rossafimohamed@gmail.com}
\section{Introduction}
Frames for Hilbert spaces were introduced in 1952 by Duffin and Schaefer \cite{Duf}. They abstracted the fundamental notion of Gabor \cite{Gab} to study signal processing. Many generalizations of frames were introduced, frames of subspaces \cite{Asg}, Pseudo-frames \cite{Li}, oblique frames \cite{OC}, g-frames \cite{AB}, $\ast$-frame \cite{Ali} in Hilbert $ C^{\ast} $-modules. In 2000, Frank-Larson \cite{F3} introduced the notion of frames in Hilbert $ C^{\ast} $-modules as a generalization of frames in Hilbert spaces. Recentely, A. Khosravi and B. Khosravi \cite{AB} introduced the g-frame theory in Hilbert $C^{\ast}$-modules, and Alijani, and Dehghan \cite{Ali} introduced the g-frame theory in Hilbert $C^{\ast}$-modules. N.Bounader and S.Kabbaj \cite{Kab} and A.Alijani \cite{A} introduced the $\ast$-g-frames which are generalizations of g-frames in Hilbert $C^{\ast}$-modules. In this article, we study the $\ast$-g-frames in tensor products of Hilbert $C^{\ast}$-modules and $\ast$-g-frames in two Hilbert $C^{\ast}$-modules with different $C^{\ast}$-algebras. In section 2 we briefly recall the definitions and basic properties of $C^{\ast}$-algebra, Hilbert $C^{\ast}$-modules, frames, g-frames, $\ast$-frames and $\ast$-g-frames in Hilbert $C^{\ast}$-modules. In section 3 we investigate tensor product of Hilbert $C^{\ast}$-modules, we show that tensor product of $\ast$-g-frames for Hilbert $C^{\ast}$-modules $\mathcal{H}$ and $\mathcal{K}$, present $\ast$-g-frames for $\mathcal{H}\otimes\mathcal{K}$, and tensor product of their $\ast$-g-frame operators is the $\ast$-g-frame operator of the tensor product of $\ast$-g-frames. We also study $\ast$-g-frames in two Hilbert $C^{\ast}$-modules with different $C^{\ast}$-algebras.
\section{Preliminaries}
Let $I$ and $J$ be countable index sets. In this section we briefly recall the definitions and basic properties of $C^{\ast}$-algebra, Hilbert $C^{\ast}$-modules, g-frame, $\ast$-g-frame in Hilbert $C^{\ast}$-modules. For information about frames in Hilbert spaces we refer to \cite{Ch}. Our reference for $C^{\ast}$-algebras is \cite{Dav,Con}. For a $C^{\ast}$-algebra $\mathcal{A}$, an element $a\in\mathcal{A}$ is positive ($a\geq 0$) if $a=a^{\ast}$ and $sp(a)\subset\mathbf{R^{+}}$. $\mathcal{A}^{+}$ denotes the set of positive elements of $\mathcal{A}$.
\begin{definition}
	\cite{BA}. Let $ \mathcal{A} $ be a unital $C^{\ast}$-algebra and $\mathcal{H}$ be a left $ \mathcal{A} $-module, such that the linear structures of $\mathcal{A}$ and $ \mathcal{H} $ are compatible. $\mathcal{H}$ is a pre-Hilbert $\mathcal{A}$-module if $\mathcal{H}$ is equipped with an $\mathcal{A}$-valued inner product $\langle.,.\rangle_{\mathcal{A}} :\mathcal{H}\times\mathcal{H}\rightarrow\mathcal{A}$, such that is sesquilinear, positive definite and respects the module action. In the other words,
	\begin{itemize}
		\item [(i)] $ \langle x,x\rangle_{\mathcal{A}}\geq0 $ for all $ x\in\mathcal{H} $ and $ \langle x,x\rangle_{\mathcal{A}}=0$ if and only if $x=0$.
		\item [(ii)] $\langle ax+y,z\rangle_{\mathcal{A}}=a\langle x,y\rangle_{\mathcal{A}}+\langle y,z\rangle_{\mathcal{A}}$ for all $a\in\mathcal{A}$ and $x,y,z\in\mathcal{H}$.
		\item[(iii)] $ \langle x,y\rangle_{\mathcal{A}}=\langle y,x\rangle_{\mathcal{A}}^{\ast} $ for all $x,y\in\mathcal{H}$.
	\end{itemize}	 
	For $x\in\mathcal{H}, $ we define $||x||=||\langle x,x\rangle_{\mathcal{A}}||^{\frac{1}{2}}$. If $\mathcal{H}$ is complete with $||.||$, it is called a Hilbert $\mathcal{A}$-module or a Hilbert $C^{\ast}$-module over $\mathcal{A}$. For every $a$ in $C^{\ast}$-algebra $\mathcal{A}$, we have $|a|=(a^{\ast}a)^{\frac{1}{2}}$ and the $\mathcal{A}$-valued norm on $\mathcal{H}$ is defined by $|x|=\langle x, x\rangle_{\mathcal{A}}^{\frac{1}{2}}$ for $x\in\mathcal{H}$.
	
	Let $\mathcal{H}$ and $\mathcal{K}$ be two Hilbert $\mathcal{A}$-modules, A map $T:\mathcal{H}\rightarrow\mathcal{K}$ is said to be adjointable if there exists a map $T^{\ast}:\mathcal{K}\rightarrow\mathcal{H}$ such that $\langle Tx,y\rangle_{\mathcal{A}}=\langle x,T^{\ast}y\rangle_{\mathcal{A}}$ for all $x\in\mathcal{H}$ and $y\in\mathcal{K}$.
	
	From now on, we assume that $\{V_{i}\}_{i\in I}$ and $\{W_{j}\}_{j\in J}$ are two sequences of Hilbert $\mathcal{A}$-modules. We also reserve the notation $End_{\mathcal{A}}^{\ast}(\mathcal{H},\mathcal{K})$ for the set of all adjointable operators from $\mathcal{H}$ to $\mathcal{K}$ and $End_{\mathcal{A}}^{\ast}(\mathcal{H},\mathcal{H})$ is abbreviated to $End_{\mathcal{A}}^{\ast}(\mathcal{H})$.
\end{definition}
\begin{definition} 
	\cite{BA}. Let $ \mathcal{H} $ be a Hilbert $\mathcal{A}$-module. A family $\{x_{i}\}_{i\in I}$ of elements of $\mathcal{H}$ is a frame for $ \mathcal{H} $, if there exist two positive constants $A$ , $B$, such that for all $x\in\mathcal{H}$,
	\begin{equation}\label{1}
		A\langle x,x\rangle_{\mathcal{A}}\leq\sum_{i\in I}\langle x,x_{i}\rangle_{\mathcal{A}}\langle x_{i},x\rangle_{\mathcal{A}}\leq B\langle x,x\rangle_{\mathcal{A}}.
	\end{equation}
	The numbers $A$ and $B$ are called lower and upper bound of the frame, respectively. If $A=B=\lambda$, the frame is $\lambda$-tight. If $A = B = 1$, it is called a normalized tight frame or a Parseval frame. If the sum in the middle of \eqref{1} is convergent in norm, the frame is called standard.
\end{definition}	
\begin{definition}
	\cite{AB}. We call a sequence $\{\Lambda_{i}\in End_{\mathcal{A}}^{\ast}(\mathcal{H},V_{i}):i\in I \}$ a g-frame in Hilbert $\mathcal{A}$-module $\mathcal{H}$ with respect to $\{V_{i}:i\in I \}$ if there exist two positive constants $C$, $D$, such that for all $x\in\mathcal{H}$, 
	\begin{equation}\label{2}
		C\langle x,x\rangle_{\mathcal{A}}\leq\sum_{i\in I}\langle \Lambda_{i}x,\Lambda_{i}x\rangle_{\mathcal{A}}\leq D\langle x,x\rangle_{\mathcal{A}}.
	\end{equation}
	The numbers $C$ and $D$ are called lower and upper bound of the g-frame, respectively. If $C=D=\lambda$, the g-frame is $\lambda$-tight. If $C = D = 1$, it is called a g-Parseval frame. If the sum in the middle of \eqref{2} is convergent in norm, the g-frame is called standard.
\end{definition}
\begin{definition}
	\cite{Ali}. Let $ \mathcal{H} $ be a Hilbert $\mathcal{A}$-module. A family $\{x_{i}\}_{i\in I}$ of elements of $\mathcal{H}$ is a $\ast$-frame for $ \mathcal{H} $, if there	exist strictly nonzero elements $A$ , $B$ in $\mathcal{A}$, such that for all $x\in\mathcal{H}$,
	\begin{equation}\label{3}
		A\langle x,x\rangle_{\mathcal{A}} A^{\ast}\leq\sum_{i\in I}\langle x,x_{i}\rangle_{\mathcal{A}}\langle x_{i},x\rangle_{\mathcal{A}}\leq B\langle x,x\rangle_{\mathcal{A}} B^{\ast}.
	\end{equation}
	The numbers $A$ and $B$ are called lower and upper bound of the $\ast$-frame, respectively. If $A=B=\lambda$, the $\ast$-frame is $\lambda$-tight. If $A = B = 1$, it is called a normalized tight $\ast$-frame or a Parseval $\ast$-frame. If the sum in the middle of \eqref{3} is convergent in norm, the $\ast$-frame is called standard.
\end{definition}
\begin{definition}
	\cite{Kab}. We call a sequence $\{\Lambda_{i}\in End_{\mathcal{A}}^{\ast}(\mathcal{H},V_{i}):i\in I \}$ a $\ast$-g-frame in Hilbert $\mathcal{A}$-module $\mathcal{H}$ with respect to $\{V_{i}:i\in I \}$ if there exist strictly nonzero elements $A$ , $B$ in $\mathcal{A}$, such that for all $x\in\mathcal{H}$, 
	\begin{equation}\label{4}
		A\langle x,x\rangle_{\mathcal{A}} A^{\ast}\leq\sum_{i\in I}\langle \Lambda_{i}x,\Lambda_{i}x\rangle_{\mathcal{A}}\leq B\langle x,x\rangle_{\mathcal{A}} B^{\ast}.
	\end{equation}
	The numbers $A$ and $B$ are called lower and upper bound of the $\ast$-g-frame, respectively. If $A=B=\lambda$, the $\ast$-g-frame is $\lambda$-tight. If $A= B = 1$, it is called a $\ast$-g-Parseval frame. If the sum in the middle of \eqref{4} is convergent in norm, the $\ast$-g-frame is called standard.\\
	The $\ast$-g-frame operator $S_{\Lambda}$ is defined by: $S_{\Lambda}x=\sum_{i\in I}\Lambda_{i}^{\ast}\Lambda_{i}x$ ,$\forall x\in\mathcal{H}$.
\end{definition}
\begin{lemma}\label{2.6}
	\cite{BA}. If $Q\in End_{\mathcal{A}}^{\ast}(\mathcal{H})$ is an invertible $\mathcal{A}$-linear map then for all $z\in\mathcal{H}\otimes\mathcal{K}$ we have $$\|Q^{\ast-1}\|^{-1}.|z|\leq|(Q^{\ast}\otimes I)z|\leq\|Q\|.|z|.$$
\end{lemma}
\begin{lemma}\label{2.7}
	\cite{Ali}. If $\varphi:\mathcal{A}\longrightarrow\mathcal{B}$ is a $\ast$-homomorphism between $\mathcal{C}^{\ast}$-algebras, then $\varphi$ is increasing, that is, if $a\leq b$, then $\varphi(a)\leq\varphi(b)$.
\end{lemma}
\section{Main results}
Suppose that $\mathcal{A} , \mathcal{B}$ are $C^{\ast}$-algebras and we take $\mathcal{A}\otimes\mathcal{B}$ as the completion of $\mathcal{A}\otimes_{alg}\mathcal{B}$ with the spatial norm. $\mathcal{A}\otimes\mathcal{B}$ is the spatial tensor product of $\mathcal{A}$ and $\mathcal{B}$, also suppose that $\mathcal{H}$ is a Hilbert $\mathcal{A}$-module and $\mathcal{K}$ is a Hilbert $\mathcal{B}$-module. We want to define $\mathcal{H}\otimes\mathcal{K}$ as a Hilbert $(\mathcal{A}\otimes\mathcal{B})$-module. Start by forming the algebraic tensor product $\mathcal{H}\otimes_{alg}\mathcal{K}$ of the vector spaces $\mathcal{H}$, $\mathcal{K}$ (over $\mathbb{C}$). This is a left module over $(\mathcal{A}\otimes_{alg}\mathcal{B})$ (the module action being given by $(a\otimes b)(x\otimes y)=ax\otimes by$ $(a\in\mathcal{A},b\in\mathcal{B},x\in\mathcal{H},y\in\mathcal{K})$). For $(x_{1},x_{2}\in\mathcal{H},y_{1},y_{2}\in\mathcal{K})$ we define $ \langle x_{1}\otimes y_{1},x_{2}\otimes y_{2}\rangle_{\mathcal{A}\otimes\mathcal{B}}=\langle x_{1},x_{2}\rangle_{\mathcal{A}}\otimes\langle y_{1},y_{2}\rangle_{\mathcal{B}} $. We also know that for $ z=\sum_{i=1}^{n}x_{i}\otimes y_{i} $ in $\mathcal{H}\otimes_{alg}\mathcal{K}$ we have $ \langle z,z\rangle_{\mathcal{A}\otimes\mathcal{B}}=\sum_{i,j}\langle x_{i},x_{j}\rangle_{\mathcal{A}}\otimes\langle y_{i},y_{j}\rangle_{\mathcal{B}}\geq0 $ and $ \langle z,z\rangle_{\mathcal{A}\otimes\mathcal{B}}=0 $ iff $z=0$.
This extends by linearity to an $(\mathcal{A}\otimes_{alg}\mathcal{B})$-valued sesquilinear form on $\mathcal{H}\otimes_{alg}\mathcal{K}$, which makes $\mathcal{H}\otimes_{alg}\mathcal{K}$ into a semi-inner-product module over the pre-$\mathcal{C}^{\ast}$-algebra $(\mathcal{A}\otimes_{alg}\mathcal{B})$.
The semi-inner-product on $\mathcal{H}\otimes_{alg}\mathcal{K}$ is actually an inner product, see \cite{Lan}. Then $\mathcal{H}\otimes_{alg}\mathcal{K}$ is an inner-product module over the pre-$\mathcal{C}^{\ast}$-algebra $(\mathcal{A}\otimes_{alg}\mathcal{B})$, and we can perform the double completion discussed in chapter 1 of \cite{Lan} to conclude that the completion $\mathcal{H}\otimes\mathcal{K}$ of $\mathcal{H}\otimes_{alg}\mathcal{K}$ is a Hilbert $(\mathcal{A}\otimes\mathcal{B})$-module. We call $ \mathcal{H}\otimes\mathcal{K} $ the exterior tensor product of $\mathcal{H}$ and $\mathcal{K}$. With $\mathcal{H}$ ,$\mathcal{K}$ as above, we wish to investigate the adjointable operators on $ \mathcal{H}\otimes\mathcal{K} $. Suppose that $S\in End_{\mathcal{A}}^{\ast}(\mathcal{H})$ and $T\in End_{\mathcal{B}}^{\ast}(\mathcal{K})$. We define a linear operator $S\otimes T$ on $ \mathcal{H}\otimes\mathcal{K} $ by $S\otimes T(x\otimes y)=Sx\otimes Ty  (x\in\mathcal{H} ,y\in\mathcal{K})$. It is a routine verification that  is $S^{\ast}\otimes T^{\ast}$ is the adjoint of $S\otimes T$ , so in fact $S\otimes T\in End_{\mathcal{A\otimes B}}^{\ast}(\mathcal{H}\otimes\mathcal{K})$. For more details see \cite{Dav,Lan}. We note that if $a\in\mathcal{A}^{+}$ and $b\in\mathcal{B}^{+}$ , then $a\otimes b\in(\mathcal{A}\otimes\mathcal{B})^{+}$. Plainly if $a$ , $b$ are Hermitian elements of $\mathcal{A}$ and $a\geq b$ , then for every positive element $x$ of $\mathcal{B}$, we have $a\otimes x\geq b\otimes x$.
\begin{theorem}
	Let $\mathcal{H}$ and $\mathcal{K}$ be two Hilbert $C^{\ast}$-modules over unitary $C^{\ast}$-algebras $\mathcal{A}$ and $\mathcal{B}$, respectively. Let $\{\Lambda_{i}\}_{i\in I}\subset End_{\mathcal{A}}^{\ast}(\mathcal{H},V_{i})$ and $\{\Gamma_{j}\}_{j\in J}\subset End_{\mathcal{B}}^{\ast}(\mathcal{K},W_{i})$ be two $\ast$-g-frames for $\mathcal{H}$ and $\mathcal{K}$ with $\ast$-g-frame operators $S_{\Lambda}$ and $S_{\Gamma}$ and $\ast$-g-frame bounds $(A,B)$ and $(C,D)$ respectively. Then $\{\Lambda_{i}\otimes\Gamma_{j}\}_{i\in I,j\in J}  $ is a $\ast$-g-frame for Hibert $\mathcal{A}\otimes\mathcal{B}$-module $\mathcal{H}\otimes\mathcal{K}$ with $\ast$-g-frame operator $ S_{\Lambda}\otimes S_{\Gamma}$ and lower and upper $\ast$-g-frame bounds $A\otimes C$ and $ B\otimes D $, respectively.
\end{theorem}
\begin{proof}
	By the definition of $\ast$-g-frames $\{\Lambda_{i}\}_{i\in I} $ and $\{\Gamma_{j}\}_{j\in J}$ we have 
	$$A\langle x,x\rangle_{\mathcal{A}} A^{\ast}\leq\sum_{i\in I}\langle \Lambda_{i}x,\Lambda_{i}x\rangle_{\mathcal{A}}\leq B\langle x,x\rangle_{\mathcal{A}} B^{\ast} , \forall x\in\mathcal{H}.$$
	$$C\langle y,y\rangle_{\mathcal{B}} C^{\ast}\leq\sum_{j\in J}\langle \Gamma_{j}y,\Gamma_{j}y\rangle_{\mathcal{B}}\leq D\langle y,y\rangle_{\mathcal{B}} D^{\ast} , \forall y\in\mathcal{K}.$$
	Therefore
	$$
	\aligned
	(A\langle x,x\rangle_{\mathcal{A}} A^{\ast})\otimes (C\langle y,y\rangle_{\mathcal{B}} C^{\ast})&\leq\sum_{i\in I}\langle \Lambda_{i}x,\Lambda_{i}x\rangle_{\mathcal{A}}\otimes\sum_{j\in J}\langle \Gamma_{j}y,\Gamma_{j}y\rangle_{\mathcal{B}}\\
	&\leq (B\langle x,x\rangle_{\mathcal{A}} B^{\ast})\otimes (D\langle y,y\rangle_{\mathcal{B}} D^{\ast}) , \forall x\in\mathcal{H} ,\forall y\in\mathcal{K}.
	\endaligned
	$$
	Then
	$$
	\aligned
	(A\otimes C)(\langle x,x\rangle_{\mathcal{A}}\otimes\langle y,y\rangle_{\mathcal{B}}) (A^{\ast}\otimes C^{\ast})&\leq\sum_{i\in I,j\in J}\langle \Lambda_{i}x,\Lambda_{i}x\rangle_{\mathcal{A}}\otimes\langle \Gamma_{j}y,\Gamma_{j}y\rangle_{\mathcal{B}}\\
	&\leq (B\otimes D)(\langle x,x\rangle_{\mathcal{A}}\otimes\langle y,y\rangle_{\mathcal{B}}) (B^{\ast}\otimes D^{\ast}) , \forall x\in\mathcal{H} ,\forall y\in\mathcal{K}.
	\endaligned
	$$
	Consequently we have
	$$
	\aligned
	(A\otimes C)\langle x\otimes y,x\otimes y\rangle_{\mathcal{A\otimes B}} (A\otimes C)^{\ast}&\leq\sum_{i\in I,j\in J}\langle\Lambda_{i}x\otimes\Gamma_{j}y,\Lambda_{i}x\otimes\Gamma_{j}y\rangle_{\mathcal{A\otimes B}} \\
	&\leq (B\otimes D)\langle x\otimes y,x\otimes y\rangle_{\mathcal{A\otimes B}} (B\otimes D)^{\ast}, \forall x\in\mathcal{H}, \forall y\in\mathcal{K}.
	\endaligned
	$$
	Then for all $x\otimes y\in\mathcal{H\otimes K}$ we have
	$$
	\aligned
	(A\otimes C)\langle x\otimes y,x\otimes y\rangle_{\mathcal{A\otimes B}} (A\otimes C)^{\ast}&\leq\sum_{i\in I,j\in J}\langle(\Lambda_{i}\otimes\Gamma_{j})(x\otimes y),(\Lambda_{i}\otimes\Gamma_{j})(x\otimes y)\rangle_{\mathcal{A\otimes B}} \\
	&\leq (B\otimes D)\langle x\otimes y,x\otimes y\rangle_{\mathcal{A\otimes B}} (B\otimes D)^{\ast}.
	\endaligned
	$$
	The last inequality is satisfied for every finite sum of elements in $\mathcal{H}\otimes_{alg}\mathcal{K}$ and then it's satisfied for all $z\in\mathcal{H\otimes K}$. It shows that $\{\Lambda_{i}\otimes\Gamma_{j}\}_{i\in I,j\in J}  $ is $\ast$-g-frame for Hibert $\mathcal{A}\otimes\mathcal{B}$-module $\mathcal{H}\otimes\mathcal{K}$ with lower and upper $\ast$-g-frame bounds $A\otimes C$ and $ B\otimes D $, respectively.\\
	By the definition of $\ast$-g-frame operator $S_{\Lambda}$ and $S_{\Gamma}$ we have:$$S_{\Lambda}x=\sum_{i\in I}\Lambda_{i}^{\ast}\Lambda_{i}x, \forall x\in\mathcal{H}.$$
	$$S_{\Gamma}y=\sum_{j\in J}\Gamma_{j}^{\ast}\Gamma_{j}y, \forall y\in\mathcal{K}.$$
	Therefore
	$$
	\aligned
	(S_{\Lambda}\otimes S_{\Gamma})(x\otimes y)&=S_{\Lambda}x\otimes S_{\Gamma}y\\
	&=\sum_{i\in I}\Lambda_{i}^{\ast}\Lambda_{i}x\otimes\sum_{j\in J}\Gamma_{j}^{\ast}\Gamma_{j}y\\
	&=\sum_{i\in I,j\in J}\Lambda_{i}^{\ast}\Lambda_{i}x\otimes\Gamma_{j}^{\ast}\Gamma_{j}y\\
	&=\sum_{i\in I,j\in J}(\Lambda_{i}^{\ast}\otimes\Gamma_{j}^{\ast})(\Lambda_{i}x\otimes\Gamma_{j}y)\\
	&=\sum_{i\in I,j\in J}(\Lambda_{i}^{\ast}\otimes\Gamma_{j}^{\ast})(\Lambda_{i}\otimes\Gamma_{j})(x\otimes y)\\
	&=\sum_{i\in I,j\in J}(\Lambda_{i}\otimes\Gamma_{j})^{\ast})(\Lambda_{i}\otimes\Gamma_{j})(x\otimes y).
	\endaligned
	$$
	Now by the uniqueness of $\ast$-g-frame operator, the last expression is equal to $S_{\Lambda\otimes\Gamma}(x\otimes y)$. Consequently we have $ (S_{\Lambda}\otimes S_{\Gamma})(x\otimes y)=S_{\Lambda\otimes\Gamma}(x\otimes y)$. The last equality is satisfied for every finite sum of elements in $\mathcal{H}\otimes_{alg}\mathcal{K}$ and then it's satisfied for all $z\in\mathcal{H\otimes K}$. It shows that $ (S_{\Lambda}\otimes S_{\Gamma})(z)=S_{\Lambda\otimes\Gamma}(z)$. So $S_{\Lambda\otimes\Gamma}=S_{\Lambda}\otimes S_{\Gamma}$.
\end{proof}
\begin{theorem}
	If $Q\in End_{\mathcal{A}}^{\ast}(\mathcal{H})$ is invertible and $\{\Lambda_{i}\}_{i\in I}\subset End_{\mathcal{A\otimes B}}^{\ast}(\mathcal{H\otimes K})$ is a $\ast$-g-frame for $\mathcal{H\otimes K}$ with lower and upper  $\ast$-g-frame bounds $A$ and $B$ respectively and $\ast$-g-frame operator $S$, then $\{\Lambda_{i}(Q^{\ast}\otimes I)\}_{i\in I} $ is a $\ast$-g-frame for $\mathcal{H\otimes K}$ with lower and upper $\ast$-g-frame bounds $\|Q^{\ast-1}\|^{-1}A$ and $\|Q\|B$ respectively and $\ast$-g-frame operator $(Q\otimes I)S(Q^{\ast}\otimes I)$.
\end{theorem}
\begin{proof}
	Since $Q\in End_{\mathcal{A}}^{\ast}(\mathcal{H})$, $Q\otimes I\in End_{\mathcal{A\otimes B}}^{\ast}(\mathcal{H\otimes K})$ with inverse $Q^{-1}\otimes I$. It is obvious that the adjoint of $Q\otimes I$ is $Q^{\ast}\otimes I$. An easy calculation shows that for every elementary tensor $x\otimes y$,
	$$
	\aligned
	\|(Q\otimes I)(x\otimes y)\|^{2}&=\|Q(x)\otimes y\|^{2}\\
	&=\|Q(x)\|^{2}\|y\|^{2}\\
	&\leq\|Q\|^{2}\|x\|^{2}\|y\|^{2}\\
	&=\|Q\|^{2}\|x\otimes y\|^{2}.
	\endaligned
	$$
	So $Q\otimes I$ is bounded, and therefore it can be extended to $\mathcal{H\otimes K}$. Similarly for $Q^{\ast}\otimes I$, hence $Q\otimes I$ is $\mathcal{A\otimes B}$-linear, adjointable with adjoint $Q^{\ast}\otimes I$. Hence for every $z\in\mathcal{H\otimes K}$ we have by lemma \ref{2.6} $$\|Q^{\ast-1}\|^{-1}.|z|\leq|(Q^{\ast}\otimes I)z|\leq\|Q\|.|z|.$$
	By the definition of $\ast$-g-frames we have 
	$$A\langle z,z\rangle_{\mathcal{A\otimes B}} A^{\ast}\leq\sum_{i\in I}\langle \Lambda_{i}z,\Lambda_{i}z\rangle_{\mathcal{A\otimes B}}\leq B\langle z,z\rangle_{\mathcal{A\otimes B}} B^{\ast}.$$
	Then 
	$$
	\aligned
	A\langle (Q^{\ast}\otimes I)z,(Q^{\ast}\otimes I)z\rangle_{\mathcal{A\otimes B}} A^{\ast}&\leq\sum_{i\in I}\langle \Lambda_{i}(Q^{\ast}\otimes I)z,\Lambda_{i}(Q^{\ast}\otimes I)z\rangle_{\mathcal{A\otimes B}}\\
	&\leq B\langle (Q^{\ast}\otimes I)z,(Q^{\ast}\otimes I)z\rangle_{\mathcal{A\otimes B}} B^{\ast}.
	\endaligned
	$$
	So$$\|Q^{\ast-1}\|^{-1}A\langle z,z\rangle_{\mathcal{A\otimes B}}(\|Q^{\ast-1}\|^{-1}A)^{\ast}\leq\sum_{i\in I}\langle \Lambda_{i}(Q^{\ast}\otimes I)z,\Lambda_{i}(Q^{\ast}\otimes I)z\rangle_{\mathcal{A\otimes B}}\leq\|Q\|B\langle z,z\rangle_{\mathcal{A\otimes B}}(\|Q\|B)^{\ast}.$$
	Now
	$$
	\aligned
	(Q\otimes I)S(Q^{\ast}\otimes I)&=(Q\otimes I)(\sum_{i\in I}\Lambda_{i}^{\ast}\Lambda_{i})(Q^{\ast}\otimes I)\\
	&=\sum_{i\in I}(Q\otimes I)\Lambda_{i}^{\ast}\Lambda_{i}(Q^{\ast}\otimes I)\\
	&=\sum_{i\in I}(\Lambda_{i}(Q^{\ast}\otimes I))^{\ast}\Lambda_{i}(Q^{\ast}\otimes I).
	\endaligned
	$$
	Which completes the proof.
\end{proof}
\begin{theorem}
	Let $(\mathcal{H},\mathcal{A},\langle.,.\rangle_{\mathcal{A}})$ and $(\mathcal{H},\mathcal{B},\langle.,.\rangle_{\mathcal{B}})$ be two Hilbert $\mathcal{C^{\ast}}$-modules and let $\varphi :\mathcal{A}\longrightarrow \mathcal{B}$ be a $\ast$-homomorphism and $\theta$ be a map on $\mathcal{H}$ such that $\langle \theta x,\theta y\rangle_{\mathcal{B}}=\varphi(\langle x, y\rangle_{\mathcal{A}})$ for all $x,y\in\mathcal{H}$. Also, suppose that $\{\Lambda_{i}\}_{i\in I}\subset End_{\mathcal{A}}^{\ast}(\mathcal{H},V_{i})$ (where $V_{i}=\mathcal{H}$ for each $i$ in $I$) is a $\ast$-g-frame for $(\mathcal{H},\mathcal{A},\langle.,.\rangle_{\mathcal{A}})$ with $\ast$-g-frame operator $S_{\mathcal{A}} $ and lower and upper $\ast$-g-frame bounds $A$ , $B$  respectively. If $\theta$ is surjective and $\theta\Lambda_{i}=\Lambda_{i}\theta$ for each $i$ in $I$, then $\{\Lambda_{i}\}_{i\in I}$ is a $\ast$-g-frame for $(\mathcal{H},\mathcal{B},\langle.,.\rangle_{\mathcal{B}})$ with $\ast$-g-frame operator $S_{\mathcal{B}} $ and lower and upper $\ast$-g-frame bounds $\varphi(A)$ ,$\varphi(B)$ respectively, and $\langle S_{\mathcal{B}}\theta x,\theta y\rangle_{\mathcal{B}}=\varphi(\langle S_{\mathcal{A}}x, y\rangle_{\mathcal{A}})$.
\end{theorem}
\begin{proof} Let $y\in\mathcal{H}$ then there exists $x\in\mathcal{H}$ such that $\theta x=y$ ($\theta$ is surjective). By the definition of $\ast$-g-frames we have
	$$A\langle x,x\rangle_{\mathcal{A}} A^{\ast}\leq\sum_{i\in I}\langle \Lambda_{i}x,\Lambda_{i}x\rangle_{\mathcal{A}}\leq B\langle x,x\rangle_{\mathcal{A}} B^{\ast}.$$
	By lemma \ref{2.7} we have
	$$\varphi(A\langle x,x\rangle_{\mathcal{A}} A^{\ast})\leq\varphi(\sum_{i\in I}\langle \Lambda_{i}x,\Lambda_{i}x\rangle_{\mathcal{A}})\leq\varphi( B\langle x,x\rangle_{\mathcal{A}} B^{\ast}).$$
	By the definition of $\ast$-homomorphism we have
	$$\varphi(A)\varphi(\langle x,x\rangle_{\mathcal{A}}) \varphi(A^{\ast})\leq\sum_{i\in I}\varphi(\langle \Lambda_{i}x,\Lambda_{i}x\rangle_{\mathcal{A}})\leq\varphi( B)\varphi(\langle x,x\rangle_{\mathcal{A}}) \varphi(B^{\ast}).$$
	By the relation betwen $\theta$ and $\varphi$ we get
	$$\varphi(A)\langle \theta x,\theta x\rangle_{\mathcal{B}} \varphi(A)^{\ast}\leq\sum_{i\in I}\langle \theta\Lambda_{i}x,\theta\Lambda_{i}x\rangle_{\mathcal{B}}\leq\varphi( B)\langle\theta x,\theta x\rangle_{\mathcal{B}} \varphi(B)^{\ast}.$$
	By the relation betwen $\theta$ and $\Lambda_{i}$ we have
	$$\varphi(A)\langle \theta x,\theta x\rangle_{\mathcal{B}} \varphi(A)^{\ast}\leq\sum_{i\in I}\langle \Lambda_{i}\theta x,\Lambda_{i}\theta x\rangle_{\mathcal{B}}\leq\varphi( B)\langle\theta x,\theta x\rangle_{\mathcal{B}} \varphi(B)^{\ast}.$$
	Then
	$$
	\varphi(A)\langle  y, y\rangle_{\mathcal{B}} (\varphi(A))^{\ast}\leq\sum_{i\in I}\langle \Lambda_{i}y,\Lambda_{i}y\rangle_{\mathcal{B}}
	\leq\varphi( B)\langle y,y\rangle_{\mathcal{B}} (\varphi(B))^{\ast} , \forall y\in\mathcal{H}.
	$$
	On the other hand we have
	$$
	\aligned
	\varphi(\langle S_{\mathcal{A}}x, y\rangle_{\mathcal{A}})&=\varphi(\langle\sum_{i\in I}\Lambda_{i}^{\ast}\Lambda_{i}x,y\rangle_{\mathcal{A}})\\
	&=\sum_{i\in I}\varphi(\langle\Lambda_{i}x,\Lambda_{i}y\rangle_{\mathcal{A}})\\
	&=\sum_{i\in I}\langle\theta\Lambda_{i}x,\theta\Lambda_{i}y\rangle_{\mathcal{B}}
	\\
	&=\sum_{i\in I}\langle\Lambda_{i}\theta x,\Lambda_{i}\theta y\rangle_{\mathcal{B}}\\
	&=\langle\sum_{i\in I}\Lambda_{i}^{\ast}\Lambda_{i}\theta x,\theta y\rangle_{\mathcal{B}}\\
	&=\langle S_{\mathcal{B}}\theta x,\theta y\rangle_{\mathcal{B}}.
	\endaligned
	$$
	Which completes the proof.
\end{proof}
In the following, we give an examples of the function $\varphi$ in the precedent theorem.
\begin{example}\label{Banach-Stone}\cite{Won}. 
	Let $X$ and $Y$ be two locally compact Hausdorff spaces. Let $H$ be a Hilbert space. Let $T$ be a surjective linear isometry from $C_{0}(X,H)$ onto $C_{0}(Y,H)$, then there exists $\phi:Y\rightarrow X$ a homeomorphism and $h(y):H\rightarrow H$ a unitary operator $\forall y\in Y$ such that $$Tf(y)=h(y)f(\phi(y)).$$
	
	In this case, we have:
	
	$$\langle Tf,Tg\rangle(y)=\langle Tf(y),Tg(y)\rangle=\langle h(y)f(\phi(y)),h(y)g(\phi(y))\rangle$$
	$$=\langle f(\phi(y)),g(\phi(y))\rangle=\langle f,g\rangle o\phi(y).$$
	Then$$\langle Tf,Tg\rangle=\langle f,g\rangle o\phi.$$
	Let $\varphi:C_{0}(X)\rightarrow C_{0}(Y)$ be the $\ast$-isomorphism defined by $\varphi(\psi)=\psi o\phi$. Then $$\langle Tf,Tg\rangle=\varphi(\langle f,g\rangle).$$
\end{example}
The example \ref{Banach-Stone} is a consequence of Banach-Stone's Theorem.
\begin{example}
	Let $\mathcal{A}$ be a $C^{\ast}$-algebra, then:
	\begin{itemize}
		\item $\mathcal{A}$ itself is a Hilbert $\mathcal{A}$-module with the inner product $\langle a,b\rangle_{r}:=a^{\ast}b,  (a,b\in\mathcal{A})$.
		\item $\mathcal{A}$ itself is a Hilbert $\mathcal{A}$-module with the inner product $\langle a,b\rangle_{l}:=ab^{\ast},  (a,b\in\mathcal{A})$.
	\end{itemize}
	Let $\theta:\mathcal{A}\rightarrow\mathcal{A}$ be the invertible map defined by $\theta(a)=a^{\ast}$ and we take $\varphi$ equal to the identity of $L(\mathcal{A})$. Then
	$$\langle\theta a,\theta b\rangle_{l}=\theta a(\theta b)^{\ast}=a^{\ast}b=\langle a,b\rangle_{r}=\varphi(\langle a,b\rangle_{r}).$$
\end{example}

\bibliographystyle{amsplain}

\end{document}